\documentclass{amsproc}
\usepackage{amsthm}
\usepackage[usenames,dvipsnames]{color}
\usepackage{amstext}
\usepackage{amssymb}
\usepackage[all]{xy}
\usepackage{hyperref}
\usepackage{amssymb}
\usepackage{mathtools}
\makeatletter
\numberwithin{equation}{section}
\numberwithin{figure}{section}
\theoremstyle{plain}
\newtheorem{thm}{\protect\theoremname}[section]
  \theoremstyle{plain}
  \newtheorem{prop}[thm]{\protect\propositionname}
\theoremstyle{plain}
\newtheorem{lemma}[thm]{\protect\lemmaname}
  \newtheorem{corollary}[thm]{\protect\corollaryname}
  \theoremstyle{definition}
  \newtheorem{question}[thm]{\protect\questionname}
   \newtheorem{example}[thm]{\protect\examplename}
   \newtheorem{definition}[thm]{\protect\definitionname}
  
\makeatother
  \providecommand{\lemmaname}{Lemma}
  \providecommand{\definitionname}{Definition}
  \providecommand{\propositionname}{Proposition}
  \providecommand{\questionname}{Question}
   \providecommand{\examplename}{Example}
  \providecommand{\theoremname}{Theorem}
  \providecommand{\corollaryname}{Corollary}
 
\begin{document}

\title[Log Centres of NCCRs are KLT]{Log Centres of Noncommutative Crepant Resolutions are
  Kawamata Log Terminal: Remarks
on a paper of Stafford and van den Bergh}

\author{Colin Ingalls}

\address{School of Mathematics and Statistics, Carleton University, Canada}

\email{cingalls@math.carleton.ca}

\author{Takehiko Yasuda}

\address{Department of Mathematics, Graduate School of Science, Osaka University,
Toyonaka, Osaka 560-0043, Japan}

\email{yasuda.takehiko.sci@osaka-u.ac.jp}

\thanks{Ingalls was supported by an NSERC Discovery Grant. Yasuda was supported by JSPS KAKENHI Grant Numbers JP22740020, JP18H01112, JP21H04994. The first author would like to thank Kenneth Brown and Michael Wemyss for helpful conversations.  We would also like to thank an anonymous referee for helpful comments.}

\begin{abstract}
  We show that if a finitely generated prime algebra $\Delta$ is a finitely generated maximal Cohen-Macaulay module over its centre $Z$, and has global dimension equal to $\dim Z$, then the pair given by its centre and ramification divisor is Kawamata log terminal.  
\end{abstract}

\maketitle

\global\long\def\AA{\mathbb{A}}
\global\long\def\PP{\mathbb{P}}
\global\long\def\NN{\mathbb{N}}
\global\long\def\GG{\mathbb{G}}
\global\long\def\ZZ{\mathbb{Z}}
\global\long\def\QQ{\mathbb{Q}}
\global\long\def\CC{\mathbb{C}}
\global\long\def\FF{\mathbb{F}}
\global\long\def\LL{\mathbb{L}}
\global\long\def\RR{\mathbb{R}}

\global\long\def\bx{\mathbf{x}}
\global\long\def\bf{\mathbf{f}}

\global\long\def\cod{\mathrm{cod}\,}
\global\long\def\gldim{\mathrm{gldim}\,}
\global\long\def\GKdim{\mathrm{GKdim}\,}
\global\long\def\Z{Z}
\global\long\def\LZ{\mathrm{logZ}}
\global\long\def\Cl{\mathrm{Cl}}
\global\long\def\cN{\mathcal{N}}
\global\long\def\cW{\mathcal{W}}
\global\long\def\cY{\mathcal{Y}}
\global\long\def\cM{\mathcal{M}}
\global\long\def\cF{\mathcal{F}}
\global\long\def\cX{\mathcal{X}}
\global\long\def\cE{\mathcal{E}}
\global\long\def\cJ{\mathcal{J}}
\global\long\def\cO{\mathcal{O}}
\global\long\def\cD{\mathcal{D}}
\global\long\def\cA{\mathcal{A}}

\global\long\def\Spec{\mathrm{Spec}\,}
\newcommand\Hom{\mathrm{Hom}}
\newcommand\End{\mathrm{End}}

\global\long\def\Sing{\mathrm{Sing}\,}
\global\long\def\Br{\mathrm{Br}\,}

\section{Introduction}

We will work over an algebraically closed field of characteristic zero.
Noncommutative crepant resolutions (NCCRs), as discussed in ~\cite{VdBICM}, indicate an interplay between noncommutative algebra and algebraic geometry.
It is not clear when a commutative algebra admits an NCCR, and various necessary conditions have been found.  In~\cite{Stafford-VdB}, it is shown that any commutative algebra admitting an NCCR  has rational singularities. In particular, they show that if an affine prime algebra $\Delta$ that is finite over its centre is homologically homogenous, or equivalently $\Delta$ is Cohen-Macaulay over its centre with finite global dimension, then its centre has rational singularities.  We use standard results about the singularities of the minimal model program~\cite{Kollar-Mori}, and the argument of ~\cite{Stafford-VdB} to show that 
the singularities of the log pair given by the centre and the ramification divisor is Kawamata log terminal.  This shows in particular that the pair is $\QQ$-Gorenstein. We now review the contents of this paper.

In Section 2, we discussion the singularities of the minimal model program and we recall the definitions and properties of log canonical covers.  In Section 3, we compute the centralizers of the canonical bimodule $\omega$ for a hereditary order over a DVR.  We also quote relevant results from~\cite{Kollar-Mori}.  This will be used in the proof of the main result.  In Section 4, we recall the properties of homologically homogeneous algebras that are finite over their centres.  This material is from~\cite{Stafford-VdB}.  
In Section 5, we prove the main result.  Given $\Delta$, a prime affine algebra that finite over its centre and homologically homogeneous, with centre $Z$ and discriminant $D$, we show that the centre of the canonical cover $\Gamma$ of the order $\Delta$ is the log canonical cover of the pair $(Z,\Delta)$.
Lastly, in Section 6, we list algebras to which the main theorem applies, and present some relevant examples.

There are many cases of singularities which are known to not admit a NCCR, for example see~\cite{Dao10}. 

The results of this paper suggest the following question.
\begin{question}  Let $Z$ be a commutative local ring.  Under what conditions is $Z$ the centre of a prime affine homologically homogeneous ring $A$?
 In other words, if  we allow the flexibility of adding a boundary divisor $D$, and allowing $A$ to have a non-trivial generic Brauer class, can we obtain a twisted ramified NCCR of the pair $(\Spec Z,D)$?
\end{question}

See~\ref{definition:twistedRamified} for the definitions of the relevant terms.
The results and a sketch of the proofs of this paper were presented in ~\cite{MSRItalk}.  The paper~\cite{DITW} works in the setting of NCCRs and Theorem 1.1(a) and Corollary 1.7 of \cite{DITW} prove our main results in that special case.

\section{The Singularities of Birational Geometry}

In this section we review some material about the singularities that
arise in the study of birational geometry and the minimal model
program.  Standard references include~\cite{ Clemens-Kollar-Mori, Kollar-Mori, Matsuki}.

\subsection{Kawamata log terminal, log terminal and canonical singularities}
Let $X$ be a normal variety.  A $\QQ$-{\it divisor} in $X$ is a linear
combination of codimension one subvarieties of $X$ with rational
coefficients.  A $\QQ$-divisor is {\it effective} if all its coefficients
are non-negative.
A $\QQ$-divisor $D$ is $\QQ$-{\it Cartier} if there is an integer $m$
so that $mD$ is a Cartier divisor.  A {\it log variety} is a pair
$(X,D)$ consisting of a normal
variety $X$ with a $\QQ$-divisor $D$.  

Let $(X,D)$ be a log variety such that $K_{X}+D$ is $\QQ$-Cartier.  In
this case we say that $(X,D)$ is $\QQ$-{\it Gorenstein.}
Then for a proper birational morphism $f:Y\to X$ with $Y$ smooth,
we write 
\[
K_{Y}=f^{*}(K_{X}+D)+\sum_{E}a_{E}E,\, a_{E}\in\QQ
\]
where $E$ runs over exceptional divisors and strict transforms of
prime divisors appearing in $D.$ The coefficient $a_{E}$ is independent
of the morphism $f$ and depends only on the divisorial valuation
of the function field associated to $E$.   The map $f:Y \to (X,D)$ is
called {\it crepant} if all $a_E=0,$ even in the case where $Y$ is only
normal and not
necessarily smooth.  We say that $(X,D)$ has
{\it Kawamata log terminal (klt) singularities} if $a_{E}>-1$ for every $E.$
When $(X,D)$ is $\QQ$-Gorenstein and $D=0,$ then $X$ is $\QQ$-Gorenstein. In this case, we say that
$X$ has {\it log terminal (resp. canonical) singularities} if $a_{E}>-1$
(resp. $\ge0$) for every $E.$  Note that $K_X$ is Cartier if and only
if $X$ is Gorenstein.  We recall the following well known results.
\begin{prop}[{\cite[Cor.~5.24]{Kollar-Mori}}]
\label{prop:rational iff canonical}Suppose that $K_{X}$ is Cartier.
Then $X$ has rational singularities if and only if $ $it has canonical
singularities.
\begin{prop}[{\cite[Th.~5.22]{Kollar-Mori}}]
\label{prop:klt implies rational}If $(X,D)$ has Kawamata log terminal
singularities with $D$ effective, then $X$ has rational singularities.
\end{prop}
\end{prop}

\subsection{Log canonical covers}

Next we recall the standard technique of log canonical covers (see
\cite[\S 8.5]{kawamata}\cite[\S 1.3]{prokhorov}).

Suppose now that a $\QQ$-Gorenstein log variety $(X,D)$ has its boundary divisor of
the form 
\[
D=\sum_{i}\frac{e_{i}-1}{e_{i}}D_{i},\, \, e_{i}\in\ZZ_{>0}
\]
with $D_{i}$ prime divisors. We say that $D$ has {\it standard
  coefficients}.  Let $m$ be the {\it index} of $(X,D),$ that
is, the least positive integer such that $m(K_{X}+D)$ is Cartier.
Given a coherent sheaf $\cF$ on $X$ we write $\cF^{(n)}:=(\cF^{\otimes
  n})^{**}$, the reflexive hull of the tensor power.
Locally on $X,$ fixing an isomorphism $\omega_{X}^{(m)}(mD)\cong\cO_{X}$,
we obtain an $\cO_{X}$-algebra 
\begin{equation}
\cA:=\bigoplus_{i=0}^{m-1}\omega_{X}^{(i)}\left(\left\lfloor iD\right\rfloor \right).\label{eq:log can cover}
\end{equation}
(In \cite{kawamata,prokhorov}, a slightly different description of the sheaf, $\bigoplus_{i=0}^{m-1}\cO_{X}\left(\left\lfloor -i(K_{X}+D)\right\rfloor \right)$,
is used, though it is isomorphic to $\cA$.)  The associated finite
morphism 
\[
\phi:\widetilde{X}:=\mathcal{S}\mathrm{pec}_{X}\,\cA\to X
\]
is called a log canonical cover and has the following properties:
\begin{prop}\cite[\S 8.5]{kawamata}\cite[\S 1.3]{prokhorov}
Let $(X,D)$ be a $\QQ$-Gorenstein log variety with log canonical cover
$\widetilde{X}$.  Then $\widetilde{X}$ is an irreducible normal variety, $K_{\widetilde{X}}$
is Cartier, and 
\[
K_{\widetilde{X}}=\phi^{*}(K_{X}+D).
\]
\end{prop}
Note that the map  $\psi : \tilde{X} \rightarrow (X,D)$
is crepant by definition.

\begin{prop}
\label{prop:cover singularities}\cite[Prop.~5.20]{Kollar-Mori} With
the notation as above, $\widetilde{X}$ has canonical singularities if
and only if $(X,D)$ has Kawamata log terminal singularities.
\end{prop}

\section{Centralizers of Powers of $\omega$}
Let $R$ be a DVR with parameter $t$. 
 We write  
$(t^i)=t^iR$ for $i \in \ZZ$,  the fractional ideal generated by $t^i$ and  
$(t^0)=(1)=R$.
Let $\Delta$ be a {\it standard hereditary order} over $R$ with ramification index $e,$ defined to be
$$\Delta = \begin{pmatrix} (1) & \cdots & \cdots & (1)\\
(t) & \ddots & \ddots & \vdots\\
\vdots & \ddots & \ddots & \vdots\\
(t) & \cdots & (t) & (1)
\end{pmatrix}
\subseteq R^{\,e \times e}.$$

The dualizing module of $\Delta$ is 
described as a bimodule by  $\omega=\Hom_R(\Delta,\omega_R)$ as described in 
\cite[Prop.~2.7]{Stafford-VdB}. We compute this bimodule 
by  forming the dual of $\Delta$ into $R$ component wise and taking the 
transpose.  Let $J$
be the Jacobson radical of $\Delta$, we have:
\[
\omega \simeq\begin{pmatrix}(1) & (t^{-1}) & \cdots & (t^{-1})\\
\vdots & \ddots & \ddots & \vdots\\
\vdots & \ddots & \ddots & (t^{-1})\\
(1) & \cdots & \cdots & (1)
\end{pmatrix}\text{ and }J = \begin{pmatrix}(t) & (1) & \cdots & (1)\\
\vdots & \ddots & \ddots & \vdots\\
\vdots & \ddots & \ddots & (1)\\
(t) & \cdots & \cdots & (t)
\end{pmatrix}.
\]
 Let
\[
y=\begin{pmatrix}0 & 1 & 0 &\cdots & 0\\
\vdots & \ddots & \ddots & \ddots & \vdots\\
\vdots & \ddots & \ddots & \ddots & 0\\
0 & \ddots & \ddots & \ddots & 1\\
t & 0 & \cdots & \cdots & 0
\end{pmatrix}
\]
Note that $y$ is a normal element in ${\Delta}$ and
$$J =y {\Delta}={\Delta} y \quad \quad y^{e}=t 
\quad \quad \omega =y^{1-e}{\Delta}=yt^{-1}{\Delta}.$$
An induction computation shows that 
\[
J^{i}=y^{i}{\Delta}=
\begin{pmatrix}
\left( t^{-\lfloor \frac{-i}{e} \rfloor} \right) & \cdots &
  \left( t^{-\lfloor \frac{e-1-i}{e} \rfloor} \right) \\
\vdots & \ddots & \vdots \\
\left( t^{-\lfloor \frac{-e+1-i}{e} \rfloor} \right) & \cdots & \ddots 
\end{pmatrix}
\]
We note that 
\[\omega^i=y^{-i(e-1)}{\Delta}=
\begin{pmatrix}
\left( t^{-\lfloor \frac{i(e-1)}{e} \rfloor} \right) & \cdots &
  \left( t^{-\lfloor \frac{(i+1)(e-1)}{e} \rfloor} \right) \\
\vdots & \ddots & \vdots \\
\left( t^{-\lfloor \frac{(i-1)(e-1)}{e} \rfloor} \right) & \cdots & \ddots 
\end{pmatrix}
\]
So the centralizer of the bimodule $\omega^i$ is 
\[ \Z(\omega^i) =\left( t^{-\lfloor \frac{i(e-1)}{e} \rfloor} \right). \]

Write $a^{m \times n}$ for the $m\times n$ matrix with entries all $a$.
We will use this notation for $a$ being an ideal or an element of a ring.

Given a sequence of natural numbers $[N]=[n_{1},\ldots,n_{e}]$, we 
let $\Delta$ be a {\it block hereditary order} given by
$$\Delta =
\begin{pmatrix}(1)^{n_{1}\times n_{1}} & \cdots & \cdots & (1)^{n_{1}\times n_{e}}\\
(t)^{n_{2}\times n_{1}} & \ddots & \ddots & \vdots\\
\vdots & \ddots & \ddots & \vdots\\
(t)^{n_{e}\times n_{1}} & \cdots & (t)^{n_{e}\times n_{e-1}} & (1){}^{n_{e}\times n_{e}}
\end{pmatrix}$$
which we
abbreviate as
\[
\Delta\simeq\begin{pmatrix} (1) & \cdots & \cdots & (1)\\
(t) & \ddots & \ddots & \vdots\\
\vdots & \ddots & \ddots & \vdots\\
(t) & \cdots & (t) & (1)
\end{pmatrix}^{[N]}\]
We note that $\Delta$ is Morita equivalent to the standard hereditary $R$-order with 
ramification index $e$.
Since we can multiply matrices by blocks, we have the following equations
$$J=(y\Delta)^{[N]}=(\Delta y)^{[N]} \quad \quad (y^e)^{[N]}=t
 \quad \quad \omega=(y^{1-e}\Delta)^{[N]} = (yt^{-1}\Delta)^{[N]}.$$
As above we get the following lemma.
\begin{lemma}\label{lem:cod1} Let $R$ be a DVR with parameter $t.$
Let $\Delta$ be a block hereditary order over with ramification index $e$.
Then the centralizer of powers of the dualizing module $\omega^i$
are given by
\[ \Z(\omega^i) =\left( t^{-\lfloor \frac{i(e-1)}{e} \rfloor} \right). \]
\end{lemma}

\section{Noncommutative Resolutions or Algebraic set up}

We firstly recall some results and notation from \cite{Stafford-VdB}.  Recall that a $k$-algebra is {\it affine} if it is finitely generated as an algebra.
We let $\Delta$ be an affine prime $k$-algebra which is finitely generated
as a module over its centre $Z:=\Z(\Delta)$. Note that $Z$ is an
integral domain~\cite[\S 10, Ex.~1]{Lam}, and we will write $K$ for its fraction
field. Let $A$ be the simple Artinian
ring of fractions of $\Delta.$ We have that $A$ is a central simple
$K$-algebra.

We say $\Delta$ is a \emph{tame }$Z$-order
if $\Delta$ is a finitely generated, reflexive $Z$ module and $\Delta_{p}$
is hereditary for all codimension one primes $p$ of $Z$. We say
that $\Delta$ is {\it homologically homogeneous} if the projective dimensions
of all simple $\Delta$ modules are equal. We recall the following
result, mostly due to Brown-Hajarnavis \cite{Brown-Hajarnavis}, which we
cite from \cite{Stafford-VdB}.
\begin{prop}
\cite[Th.~2.3]{Stafford-VdB} Let $\Delta$ be an affine homologically homogeneous
order of dimension $d$ with centre $Z$, then
\begin{enumerate}
\item $\Delta$ is a Cohen-Macaulay module over $Z.$
\item $\gldim \Delta=\GKdim\Delta=d$.
\item $Z$ is an affine Cohen-Macaulay normal domain.
\item $\Delta$ is a tame $Z$-order.
\end{enumerate}
\end{prop}
Note that the third conclusion depends on our base field $k$ having
characteristic zero.

We will now define the discriminant of a tame order in codimension one. Let
$p$ be a codimension one prime of $Z$. Since $\Delta_p$ is hereditary,
there is an \'etale extension $R$ of the DVR $Z_{p}$ so that $\Delta_{R}:=\Delta\otimes R$
is isomorphic to 
\[
\Delta_{R}\simeq\begin{pmatrix}(1)^{n_{1}\times n_{1}} & \cdots & \cdots & (1)^{n_{1}\times n_{e_{p}}}\\
(t)^{n_{2}\times n_{1}} & \ddots & \ddots & \vdots\\
\vdots & \ddots & \ddots & \vdots\\
(t)^{n_{e_{p}}\times n_{1}} & \cdots & (t)^{n_{e_{p}}\times n_{e_{p}-1}} & (1){}^{n_{e_{p}}\times n_{e_{p}}}
\end{pmatrix}
\]
where $t$ is a parameter of $R$ ~\cite{Artin,Reiner}. We will write $[N]=[n_{1},\ldots,n_{e_p}]$
and abbreviate as
\[
\Delta_{R}\simeq\begin{pmatrix} (1) & \cdots & \cdots & (1)\\
(t) & \ddots & \ddots & \vdots\\
\vdots & \ddots & \ddots & \vdots\\
(t) & \cdots & (t) & (1)
\end{pmatrix}^{[N]} 
\]
We call $e_{p}$ the {\it ramification index} of $\Delta$ at $p$.
We define
the {\it discriminant} of $\Delta$ to be the $\QQ$-divisor 

\[
D=\sum_{\begin{array}{c}
p\subset Z\\
\cod  p=1
\end{array}}\left(\frac{e_{p}-1}{e_{p}}\right)D_{p}.
\]
Note that this sum is finite since $\left\lceil D\right\rceil $ is the
classical discriminant
 of the order $\Delta$ which is defined via the reducued trace pairing~\cite[\S 25]{Reiner}.
 We will
combine the centre of $\Delta$ with its discriminant into the {\it log
centre,} $\LZ(\Delta)=(\Spec \Z(\Delta),D)$.

Suppose $\Delta$ is an affine prime algebra that is a finite module over its centre.  We have~\cite[Lemma 2.4]{Stafford-VdB} that $\Delta$ is homologically homogeneous
if and only if $\Delta$ is a Cohen-Macaulay $R$-module and $\GKdim \Delta = \gldim \Delta$.

\begin{definition} \label{definition:twistedRamified} Let $\Delta$ be a prime, affine $k$-algebra, module finite over its centre.  If $\Delta$ is homologically homogeneous, or equivalently
  by~\cite[Lemma~2.4]{Stafford-VdB}, the algebra $\Delta$ is a Cohen-Macaulay with $\GKdim \Delta = \gldim \Delta$, then we say that $\Delta$ is a {\it twisted ramified noncommutative crepant resolution (NCCR) of the pair} $\LZ(\Delta) = (X,D).$  If furthermore
  $D = 0$ then we say the NCCR is {\it unramified.} Let $K$ be the fraction field of the centre $Z$.
The twist of a twisted NCCR $\Delta$ is the class of $\Delta\otimes_Z K$ in $\Br K.$  If this is trivial, then we say $\Delta$ is an untwisted NCCR of the pair $\LZ(\Delta)$.  A {\it classical} NCCR is untwisted and unramified, and this the usual defintion of an NCCR in the literature.
  \end{definition}

Note that in the case that $Z(\Delta)$ is Gorenstein, our definition of a classical NCCR is equivalent to the definition of an NCCR in~\cite[Defn.~2.2]{VdBICM} and our definition of twisted unramified NCCR is equivalent to the definition of a twisted NCCR in~\cite[Defn.~2.2]{VdBICM}.

The next statement is well known.
\begin{prop} If $\Delta$ is a maximal order and a classical NCCR of the pair $(Z,D)$ then $D=0$ and there is a reflexive $Z$-module $M$ so that $\Delta \simeq \End_Z(M)$.
\end{prop}
\begin{proof}
  Since the generic Brauer class of $\Delta$ is trivial and $Z$ is a normal integral domain, the above statement is true in codimension one, and so taking the reflexive hull gives us the result.
\end{proof}

\section{Noncommutative Crepant Resolutions and Canonical Covers}
\begin{prop}
\label{prop:center omega}Let $\Gamma=\Delta\oplus\omega_{\Delta}\oplus\cdots\oplus\omega_{\Delta}^{( r-1)}$
be the canonical cover of $\Delta$ and put $\widetilde{Z}:=Z(\Gamma)$
and $\widetilde{X}:=\Spec{}\widetilde{Z}.$ Then the morphism $\widetilde{X}\to X$
is a log canonical cover of $(X,D)$. \end{prop}
\begin{proof}
Let $p$ be a codimension one
prime of $Z$. By \cite{Artin}, there is an \'etale extension $R$
of $\Delta_{p}$ so that $\Delta_{R}:=\Delta\otimes R$ is isomorphic
to $$ \Delta_{R}\simeq\begin{pmatrix}R^{n_{1}\times n_{1}} & \cdots & \cdots & R^{n_{1}\times n_{k}}\\
(tR)^{n_{2}\times n_{1}} & \ddots & \ddots & \vdots\\
\vdots & \ddots & \ddots & \vdots\\
(tR)^{n_{k}\times n_{1}} & \cdots & (tR)^{n_{k-1}\times n_{k}} & R{}^{n_{k}\times n_{k}}
\end{pmatrix} $$
where $t$ is a parameter of the DVR $R$.  Now Lemma~\ref{lem:cod1} shows that 
\[\Z(\omega_{\Delta_R}^i) =\left( t^{-\lfloor \frac{i(e-1)}{e} \rfloor} \right)\]
and so \[
Z(\omega_{\Delta}^{( i)})=\omega_{X}^{(i)}(\left\lfloor iD\right\rfloor ).
\] \'etale locally in codimension one. Both sides are naturally isomorphic at the generic point, so we obtain this isomorphism globally since both sides are reflexive.  It follows immediately that
$$\widetilde{Z} = \Z(\Gamma) = \bigoplus_{i=0}^{m-1} \Z(\omega_{\Delta}^{( i)})= \bigoplus_{i=0}^{m-1}\omega_{X}^{(i)}(\left\lfloor iD\right\rfloor )$$
where $m$ is the index of $K_X+D$.

\end{proof}
Let $\Delta$ be an affine prime $k$-algebra that is module finite over its centre.  Suppose further that $\Delta$ is homologically homogeneous. Let
$Z:=\Z(\Delta)$ its centre and $X:=\Spec{}Z$ its spectrum. Stafford
and Van den Bergh proved an important result about the possible singularities
of the centre:
\begin{thm}[\cite{Stafford-VdB}]
\label{thm:stafford-vdb} $X$ has only rational singularities.
\end{thm}
In fact, we can prove the following along the same line, which is
a refinement of their result because of Proposition \ref{prop:klt implies rational}.
\begin{thm}\label{thm:main}
Let $D$ be the discriminant divisor on $X$ associated to $\Delta.$
Then the log variety $(X,D)$ has Kawamata log terminal singularities.
Moreover, if $X$ is $\QQ$ -Gorenstein, then $X$ has log terminal
singularities. \end{thm}
\begin{proof}
Since the proof is local on $X,$ we may suppose that if $m$ is the
index of $(X,D),$ then $\omega_{X}^{(m)}(mD)$ is isomorphic to $Z$.
Then putting $\widetilde{Z}:=\bigoplus\omega_{X}^{(i)}(\left\lfloor iD\right\rfloor ),$
we can construct a log canonical cover of $(X,D),$ 
\[
\phi:\widetilde{X}=\Spec{}\widetilde{Z}\to X=\Spec{}Z
\]
as in (\ref{eq:log can cover}). On the other hand, the isomorphism
$\omega_{X}^{(m)}(mD)\cong Z$ defines an isomorphism $\omega_{\Delta}^{(m)}\cong\Delta$
as $\Delta$-bimodules and we have a canonical cover of $\Delta,$
\[
\Gamma:=\Delta\oplus\omega_{\Delta}\oplus\cdots\oplus\omega_{\Delta}^{(m-1)}.
\]
We show in \ref{prop:center omega} that $Z(\omega_{\Delta}^{(i)})=\omega_{X}^{(i)}(\left\lfloor iD\right\rfloor )$
and $Z(\Gamma)=\widetilde{Z}.$

Since~\cite[Prop.~3.1(1)]{Stafford-VdB} shows that $\Gamma$ is also homologically homogeneous, so by Theorem \ref{thm:stafford-vdb},
$\widetilde{Z}$ has rational singularities. Now, from Propositions \ref{prop:rational iff canonical}
and \ref{prop:cover singularities}, the pair $(X,D)$ has Kawamata log
terminal singularities. The last assertion is straightforward.
\end{proof}
We note the following corollaries.
\begin{corollary}
  Let $A$ be an unramified NCCR of a commutative algebra $R$.  Then $R$ is log terminal and $\QQ$-Gorenstein.
\end{corollary}
\begin{corollary}
  Let $\mathcal{A}$ be a (twisted, ramified) NCCR of the pair $(X,D)$.  Then the pair is Kawamata log terminal and $\QQ$-Gorenstein.
\end{corollary}

We present an alternate proof of the theorem above.  First note that:

\begin{prop}\label{prop:crepant}
  Let $\Delta$ be an affine, prime, finite over centre, homologically homogeneous $k$-algebra.  Let $\Gamma$ be the canonical cover of $\Delta$ as in ~\ref{prop:center omega}.
  Let $(X,D)$ the the log centre of $\Delta$ and let $\widetilde{X} = \Spec Z(\Gamma)$.  Then
we have that $K_{\widetilde{X}}=\pi^{*}(K_{X}+D)$. \end{prop}
\begin{proof} Let $m$ be the index of $K_X+D$ and
  let $\hat{\otimes}$ denote the reflexive hull of the tensor product.
  We have the following isomorphisms:
  $$ \Gamma \hat{\otimes}_\Delta \omega_\Delta \simeq \omega_\Gamma$$
  $$ \omega_\Gamma \simeq \omega_{\widetilde{X} }\hat{\otimes}_{\widetilde{X}} \Gamma$$
  $$ \omega_\Delta^{(m)} = \omega_{X}^{(m)}(mD)\otimes_X \Delta.$$
In other words if we consider the 
following diagram:

\[
\xymatrix{\Spec{}\Delta\ar[d] & \Spec{}\Gamma\ar[l]\ar[d]\\
(X,D) & \widetilde{X}\ar[l]
}
\]

We know that all maps except the bottom one are crepant. Hence the
bottom map is crepant as well.\end{proof}

Using Prop.~\ref{prop:crepant} we obtain another proof of Theorem~\ref{thm:main}.
\begin{proof}
  Since $\Gamma$ is homologically homogeneous with $Z(\Gamma) = \widetilde{Z}$,
by Thm.~\ref{thm:stafford-vdb}, we know that $\Spec \widetilde{Z} = \widetilde{X}$ is has rational singularities. Therefore $\widetilde{X}$ is canonical
by Prop.~\ref{prop:rational iff canonical} and hence log terminal. By Prop.~\ref{prop:crepant} we have that $K_{\widetilde{X}}=\pi^{*}(K_{X}+D).$
So by Prop.~\ref{prop:cover singularities}, we have that $(X,D)$ is
Kawamata log terminal. In particular if $X$ is $\QQ-$Gorenstein then it is
log terminal.
\end{proof}

\section{Corollaries and Examples}

Some algebras are well known to be homologically homogeneous and so provide examples of twisted and ramified NCCRs.  We would like to thank Kenneth Brown for the references for the following statement.

\begin{prop}
  Suppose that $\Delta$ is finite over its centre and $\Delta$ is one of
  \begin{enumerate}
  \item a Hopf algebra over an algebraically closed field.
  \item a connected graded affine noetherian algebra of finite global dimension.
\item a not necessarily connected Artin-Schelter regular algebra.
  \end{enumerate}
  then $\Delta$ is a (twisted, ramified) NCCR of the pair $(Z,D) = \LZ(\Delta)$ which is Kawamata log terminal and $\QQ$-Goreinstein.
  \end{prop}
\begin{proof}
The first case is ~\cite[Prop.~1.6]{Brown-Goodearl}.
The second case is in~\cite{Stafford-Zhang}.
The third case is in~\cite{Reyes-Rogalski}.
  \end{proof}

\begin{corollary}
  The log centres of the following algebras are Kawamata log terminal.
  \begin{itemize}
  \item Quantised function algebras of semisimple groups $\cO_q(G)$ when $q$ is a root of unity.
    \item  Quantised enveloping algebras of semisimple Lie algebras $U_q(\mathfrak{g})$ when $q$ is a root of unity.
  \end{itemize}
\end{corollary}

Next, we produce an example of a variety $X$ which is not $\QQ$-Gorenstein, but we add an appropriate boundary divisor $D$, we obtain a $\QQ$-Goreinstein pair $(X,D)$ which is $\QQ$-Goreinstein and admits a ramified NCCR.

\begin{example}
  Consider the quiver with vertices $Q_0 =\{ 0,1 \}$ with two arrows $u,v$ from $0$ to $1$ and two arrows $\bar{u},\bar{v}$ from $1$ to $0$.
  We let $A$ be the path algebra of this quiver with relations
  $$u\bar{u}v = v \bar{u} u \quad u\bar{v}v = v \bar{v} u $$ $$
  \bar{u} u \bar{v} =  \bar{v} u\bar{u}, \quad  \bar{u} v \bar{v} =  \bar{v} v \bar{u}.$$  Consider the commutative ring and fractional ideals
  $$R =k[a,b,c,d]/(ad-bc)$$
  $$ I = aR+cR \quad \quad
   I^* = R+ \frac{b}{a} R.$$
  Note also that $$A \simeq \begin{pmatrix} R & I \\ I^* & R \end{pmatrix}$$
  under the isomorphism
  $$ u \mapsto \begin{pmatrix} 0 & a \\  0 & 0 \end{pmatrix} \quad \quad
  v \mapsto \begin{pmatrix} 0 & c \\  0 & 0 \end{pmatrix} $$
  $$ \bar{u} \mapsto \begin{pmatrix} 0 & 0 \\  1 & 0 \end{pmatrix} \quad \quad
  \bar{v} \mapsto \begin{pmatrix} 0 & 0 \\  \frac{b}{a} & 0 \end{pmatrix}.
  $$
  It is well known that $A$ is Azumaya except at the origin, and is an
  NCCR of $R$.
  We consider an action of $\ZZ/2$ on $A$ that fixes the vertices and acts as
  $-1,1,1,1$ on the arrows $u,v,\bar{u},\bar{v}$ respectively.
  Note that the skew group ring $$A \rtimes \ZZ/2$$  has global dimension three and is homologically homogeneous.  Hence it is a ramified NCCR of its log centre.

  We compute
  $$ Z(A) = R = k[a,b,c,d](ad-bc)$$
  $$ S = Z(A\rtimes \ZZ/2) = Z(A)^{\ZZ/2} = k[a^2,ab,b^2,c,d] \subset R.$$
  The referee noted 
  $S \simeq k[x,y,z,c,d]/(xz-y^2,xd-yc,yd-zc)$ which is the base of the Francia flip.  So there is a small resolution $f:Y \to \Spec S$ given by blowing up the ideal $(c,d)$.  So $K_Y$ must be ample or anti-ample, and $K_Y = f^*K_S.$  For more details see~\cite[(2$\cdot$1),Lem.~2]{Brown} or~\cite[(3.8)]{Reid}.
  If $S$ is $\QQ$-Gorenstein then $nK_S = 0$ for some $n$, giving a
  contradiction.  We also give a toric computation below which shows that $S$ is not $\QQ$-Gorenstein.

  We write $N = \ZZ^3$ and $N' = \ZZ^2 \oplus \frac{1}{2}\ZZ$ and let
  $M = N^\vee, M' = N'^\vee$ be the dual lattices. Let $\sigma$ is the cone
  generated by the columns of the matrix:
  $$\begin{pmatrix}
    0 & 0 & 1 & 1 \\
    0 & 1 & 0 & 1\\
    1 & 1 & 1 & 1
  \end{pmatrix}.$$  We will see that the map of toric varieties determined by
  $\sigma \subset N$ to $\sigma \subset N'$ is the map $\Spec R \to \Spec S$.
We write
  $$k[M] = k[x^{\pm 1},y^{\pm 1},z^{\pm 1}] \subset
  k[x^{\pm 1},y^{\pm 1},z^{\pm 2}] = k[M'].$$
  We have
  $$ R = k[a,b,c,d] = k[zx^{-1},zy^{-1},y,x] =k[\sigma^\vee] \subset k[M'].$$
  $$ S = k[a^2,ab,b^2,c,d] = k[z^2x^{-2},z^2x^{-1}y^{-1},z^2y^{-2},y,x] =k[\sigma^\vee] \subset k[M'].$$
 
  We compute that $A\rtimes \ZZ/2$ is ramified with index two on the toric divisor $D_\rho$ corresponding to the ray generated by $\rho = (0,0,1/2) \in \sigma \subset N'$.
  Hence the pair $(\Spec S,\frac{1}{2} D_\rho)$ is $\QQ$-Gorenstein and admits a ramified NCCR.
 If we write the primitive generators of the rays of $\sigma$ in terms of the basis
  $e_1,e_2,\frac{1}{2}e_3$ of $N',$ we obtain the columns of the following matrix
  $$C = \begin{pmatrix}
    0 & 0 & 1 & 1 \\
    0 & 1 & 0 & 1 \\
    1 & 2 & 2 & 2
  \end{pmatrix}.$$
  Let $D_1,\ldots,D_4$ be the toric Weil divisors corresponding to the rays.
  A toric Weil divisor $n_1D_1+n_2D_2+n_3D_3+n_4D_4$ is $\QQ$-Cartier if and only
  if $(n_1,\ldots,n_4)$  is in the $\QQ$ span of the rows of the matrix $B$.
  It is clear that $K_{\Spec S} = -D_1-D_2-D_3-D_4$ is not in the row span of $C$. However as above we have $\frac12 D_\rho = \frac12 D_1$ then
  $$K+B = -\frac12 D_1-D_2-D_3-D_4$$ is in the row span of $C$, verifying that the pair  $(\Spec S,\frac12 D_\rho)$ is $\QQ$-Gorenstein.
  So we see that  the toric variety $\Spec S$ is not $\QQ$-Goreinstein and does
  not admit an unramified NCCR, and there is a ramified  NCCR of the pair $(\Spec S,\frac12 D_\rho)$.

\end{example}  
In the next example, the singularity $X$ is well known to not have a classical NCCR~\cite[Thm.~3.1,Rmk.~3.2]{Dao10}, but with an appropriate divisor $D$, we obtain a twisted ramified NCCR of the pair $(X,D)$.  This shows that the pair is Kawamata log terminal and $\QQ$-Gorenstein.  We would like to thank Michael Wemyss for providing us with this example.

\begin{example}
  Let $A$ be the $k$-algebra generated by $a,b,c$ with relations
  $$ac+ca,\quad bc+cb, \quad ab-ba -2c^3.$$
  The centre $Z(A)$ is generated by
  $$ x := a^2,\quad y := b^2, \quad z:=ab+ba, \quad t := c^2.$$
  One may verify that
  $$z^2-4xy = (ab-ba)^2 = 4c^6 = 4t^3.$$
  So the centre has the compound Du Val singularity with a quadric cone in its resolution~\cite[Fig.~8-3-5]{Matsuki}.  This singularity does not have a noncommutative crepant resolution~\cite[Thm.~3.1,Rmk.~3.2]{Dao10}.

  The resolution
  $$ 0 \to A \stackrel{\begin{pmatrix} b & a & c \end{pmatrix}}{\longrightarrow}
  A^3 \stackrel{\begin{pmatrix} -c & 0 & -a \\ 0 & c & b \\ -b & a & -2c^2
  \end{pmatrix}}{\longrightarrow}
  A^3 \stackrel{\begin{pmatrix} a \\ b \\ c \end{pmatrix}}{\to}
  A \to k \to 0$$
  shows that $A$ is Artin-Schelter regular and so is homologically homogeneous.
  Lastly, we compute that $A$ is ramified on the divisor $B$ defined by $t=0.$
  So if we let $X= \Spec \Z(A)$ and $D=\frac12 B$ we have that $(X,D)$ has a twisted noncommutative crepant resolution and so is Kawamata log terminal and $\QQ$-Gorenstein.

  We also point out that $A$ can be viewed as a Clifford algebra. Let
  $$Z =k[x,y,z,t]/(z^2-4xy-4t^3)$$
  and define a quadratic form $Q$ on $Z$ given by the Gram matrix
  $$ \begin{pmatrix} 2x & z & 0 \\
    z & 2y & 0 \\
    0 & 0 & 2t
  \end{pmatrix}.$$
  Write $\Cl(Q)$ for the Clifford algebra of this form, generated by $a,b,c$ as an $Z$-algebra with relations
  $$2a^2=2x,\quad ab+ba = z, \quad 2b^2 =2y, \quad ac+ca = bc+cb = 0, \quad 2c^2 =2t.$$
  Note that in this algebra we have $(ab-ba)^2 = (2c^3)^2$ giving
  a prime ideal generated by $(ab-ba-2c^3)$ so we have
   $$A = \Cl(Q)/(ab-ba-2c^3).$$  This description can be used to show that $A$ is non-trivially twisted.
  \end{example}

\bibliography{bibs}{}

\bibliographystyle{amsplain}

\end{document}